\documentclass[12pt]{amsart}

\usepackage{mathrsfs}
\usepackage{amsfonts}
\usepackage{amssymb}
\usepackage{amsfonts, amscd, amsmath, mathrsfs, amssymb, amsthm, amsxtra, bbding, epsfig, graphicx, latexsym, url, cite}
\usepackage[papersize={8.3in,11.8in},textwidth=15.6cm,textheight=24.1cm,centering]{geometry}
\usepackage{enumerate}
\usepackage{caption} 

 \usepackage{datetime}

\usepackage{graphicx}
\usepackage{tikz}
\usepackage{float} 

\usepackage{todonotes}

\usepackage{xcolor}
\definecolor{cite}{rgb}{0.00,0.60,1.00}
\definecolor{url}{rgb}{1.00,0.10,0.80}
\definecolor{link}{rgb}{0.00,0.00,1.00}
\usepackage[colorlinks,linkcolor=link,urlcolor=url,citecolor=cite,pagebackref,breaklinks]{hyperref}

\hypersetup{
pdfstartpage=1,
pdfstartview=FitH}

\def\\{\cr}
\def\({\left(}
\def\){\right)}
\def\fl#1{\left\lfloor#1\right\rfloor}

\DeclareFontFamily{U}{mathx}{\hyphenchar\font45}
\DeclareFontShape{U}{mathx}{m}{n}{
      <5> <6> <7> <8> <9> <10>
      <10.95> <12> <14.4> <17.28> <20.74> <24.88>
      mathx10
      }{}
\DeclareSymbolFont{mathx}{U}{mathx}{m}{n}
\DeclareMathAccent{\widecheck}{\mathalpha}{mathx}{"71}

\numberwithin{equation}{section}

\allowdisplaybreaks

\newtheorem{theorem}{Theorem}[section]
\newtheorem{lemma}[theorem]{Lemma}

\newtheorem{corollary}[theorem]{Corollary} 

\usepackage[norefs,nocites]{refcheck}

\makeatletter
\newcounter{roem}
\renewcommand{\theroem}{\Roman{roem}}

\newcommand{\c@org@eq}{}
\let\c@org@eq\c@equation
\newcommand{\org@theeq}{}
\let\org@theeq\theequation

\newcommand{\setroem}{
\let\c@equation\c@roem
 \let\theequation\theroem}

\newcommand{\setarab}{
\let\c@equation\c@org@eq
\let\theequation\org@theeq}
\makeatother

\newtheorem*{claim*}{Claim}

\newcommand{\Disc}{\mathrm{Disc}}
\newcommand{\ord}{\mathrm{ord}}

\usepackage{scalerel}

\renewcommand{\geq}{\geqslant}
\renewcommand{\leq}{\leqslant}
\renewcommand{\ge}{\geqslant}
\renewcommand{\le}{\leqslant}

\DeclareMathOperator{\Mod}{mod}

\renewcommand{\bmod}[1]{\,(\Mod{ #1})}

\newcommand{\A}{\mathbb{A}}
\newcommand{\C}{\mathbb{C}}
\newcommand{\F}{\mathbb{F}}
\newcommand{\N}{\mathbb{N}}

\newcommand{\Z}{\mathbb{Z}}

\newcommand{\bm}{\mathbf{m}}

\newcommand{\bu}{\mathbf{u}}
\newcommand{\bv}{\mathbf{v}}
\newcommand{\bw}{\mathbf{w}}
\newcommand{\bx}{\mathbf{x}}

\def\vec#1{\mathbf{#1}}

\newcommand{\sA}{\mathscr{A}}
\newcommand{\sF}{\mathscr{F}}
\newcommand{\sH}{\mathscr{H}}
\newcommand{\sM}{\mathscr{M}}

\newcommand{\cA}{\mathcal{A}}
\newcommand{\cB}{\mathcal{B}}

\newcommand{\cD}{\mathcal{D}}
\newcommand{\cE}{\mathcal{E}}
\newcommand{\cF}{\mathcal{F}}
\newcommand{\cG}{\mathcal{G}}

\newcommand{\cI}{\mathcal{I}}

\newcommand{\cS}{\mathcal{S}}
\newcommand{\cT}{\mathcal{T}}
\newcommand{\cU}{\mathcal{U}}

\newcommand{\cX}{\mathcal{X}}
\newcommand{\cY}{\mathcal{Y}}

\newcommand{\fM}{\mathfrak{M}}

\newcommand{\fS}{\mathfrak{S}}

\newcommand \fU{\mathfrak U}
\newcommand\fs{\mathfrak s}
\newcommand \vm{\vec{m}}
\newcommand \vu{\vec{u}}
\newcommand \vv{\vec{v}}
\newcommand \vx{\vec{x}}

\def\mand{\qquad\mbox{and}\qquad}

\def\le{\leqslant}
\def\leq{\leqslant}
\def\ge{\geqslant}
\def\geq{\geqslant}

\usepackage{graphicx}
\usepackage{tikz}

\begin{document}


\title{Short sums of the M{\"o}bius function over function fields}

\author[\'E. Fouvry] {\'Etienne Fouvry}
\address{\'EF: D{\'e}partement de Math{\'e}matiques, Universit{\'e} Paris-Saclay,
91405 Orsay Ce\-dex, 
France}
\email{etienne.fouvry@universite-paris-saclay.fr}

\author[I. E. Shparlinski] {Igor E. Shparlinski}
\address{IES: School of Mathematics and Statistics, University of New South Wales, Sydney NSW 2052, Australia}
\email{igor.shparlinski@unsw.edu.au}
 
\author[P. Xi]{Ping Xi}
\address{PX: School of Mathematics and Statistics, Xi'an Jiaotong University, Xi'an 710049, P. R. China}
\email{ping.xi@xjtu.edu.cn}

\begin{abstract} 
We develop a new method which allows us to estimate, with a power saving, 
sums of the M{\"o}bius function over function fields $\F_p(X)$. The sums run over monic polynomials with all other coefficients in short intervals, of length below the P\'olya--Vinogradov threshold $p^{1/2}$. A classical result of L.~Stickelberger (1898) reduces this question to bounding character sums with polynomial discriminants which are inhomogeneous multivariate polynomials. Variants of Burgess's method with non-linear shifts {\it {\`a} la carte} are introduced to go beyond the P\'olya--Vinogradov range, combined with a stratification result of J. Xu (2020) for complete character sums.
\end{abstract}

\keywords{Geometric M{\"o}bius function, character sum, discriminants of polynomials}
\subjclass[2020]{11L40, 11T06}

\date{\today \quad \currenttime}

\maketitle

\tableofcontents


\section{Introduction and statement of the results} 

\subsection{Motivations and backgrounds}
Let $p$ be a prime and $\F_p$ the field of $p$ elements. We may identify $\F_p$ with $\{0,1,\cdots,p-1\}$.
For a monic polynomial $f(X)\in \F_p[X]$, which factors into $s$ irreducible factors over $\F_p$, we define
\[
\mu_p (f)=
\begin{cases}
0, & \text{ if } f \text{ has a multiple root in }\overline{\F}_p , \\
(-1)^s,&\text{ otherwise}, 
\end{cases}
\]
where $\overline{\F}_p$ is the algebraic closure of $\F_p$. 
This is the so-called {\it geometric M{\"o}bius function\/} over the function field $\F_p(X)$.
Motivated by sums of the classical  M{\"o}bius function over $\Z$, it is natural to study the average
\begin{equation}\label{eq:Mobiussum}
\fM (\cF) =\sum_{f\in\cF}\mu_p(f),
\end{equation}
where $\cF$ is a chosen finite family of polynomials. 
A typical goal of studying $\fM (\cF)$ is to obtain non-trivial bounds with the set $\cF$ being as sparse as possible. Motivated by Davenport's theorem and Chowla's conjecture over $\Z$, one can also study a generalised version of~\eqref{eq:Mobiussum} twisted by additive or multiplicative characters, or shifted multiplicative functions. See, for instance,~\cite{Ba26,BCJ15,BL19,CR14,DOS24,HM19,KR16,Por18,Ro18,Sa21,Sa24,SS22a,SS22b} and references therein.
Roughly speaking, many of these works examine the randomness of $\mu_p(f)$ when $f$ is controlled by its degree. More precisely, one takes the above family $\cF$ to be 
\begin{equation}\label{eq:F0}
\cF_0=\{f \text{ monic over }\F_p:~\deg f=n,~\deg(f-F)\leqslant d\},
\end{equation}
where $F$ is a fixed monic polynomial over $\F_p$ of degree $n$ and $d\leqslant n$. It is worthwhile to mention that Sawin~\cite{Sa21} has estimated $\fM(\cF_0)$ non-trivially for all $1\leqslant  d<n$, as $p\rightarrow+\infty$.

In this paper, we work with $\fM(\cF)$ in a different and substantially less studied scenario 
where the polynomials in $\cF$ are measured by the size of their coefficients. 
While our method works in greater generality, we specialise $\cF$ to be one of the  options of the form
 \begin{equation} 
\label{eq:F1}
\cG=\{X\mapsto X^d+a_1X^{d-1}+\cdots+a_{d-1}X+a_d:~(a_1,\cdots,a_d)\in\cB_d\}, 
\end{equation}
and
 \begin{equation} 
\label{eq:F2}
\cT=\{X\mapsto X^d+aX^k+b:~(a,b)\in \cB_2\},
\end{equation} 
where $\cB_\nu$ is a certain box in $\F_p^\nu$, $\nu =d,2$, with side lengths $H_1, \ldots, H_\nu$. 
We note that both sums $\fM (\cG)$ and $\fM(\cT)$ have already been considered in~\cite{DOS24}, however, for the applications in~\cite{DOS24}, 
the strength of the bound, rather than the range where it is non-trivial, is more important. Here we concentrate 
on the case when all coefficients of the polynomial fall into small boxes, with all side lengths falling below $p^{1/2}$. That is, we aim to obtain power-saving upper bounds for 
$\fM (\cG)$ and $\fM(\cT)$ beyond the so-called 
{\it  P\'olya-Vinogradov\/} range (see~\eqref{eq:PVrange} below).

\subsection{Bounds for the sums of the M{\"o}bius function over function fields}
First we consider sums over monic polynomials with all other coefficients running in ``short intervals''.

\begin{theorem}\label{thm:Mobius-general}
For each $d\geqslant2$, and for all $N_1, \ldots, N_d\in \Z$  and $\gamma>0$, put 
\begin{align*}
\cG=\{X\mapsto X^d+a_1X^{d-1}&+\cdots+a_{d-1}X+a_d:\\
& ~N_j<a_j\leqslant N_j+ p^\gamma\text{ for all }j=1,2,\cdots,d\}.
\end{align*}
For any $\gamma$ with
\[
\frac{1}{2}-\frac{1}{2(d^2+1)}<\gamma<1,
\]
there exists some $\delta>0$, depending only on $d$ and $\gamma$, such that
\[
\fM (\cG) \ll  p^{d\gamma -\delta},
\]
where the implied constant depends only on  $d$ and $\gamma$. 
\end{theorem}

In particular, one may easily see 
that Theorem~\ref{thm:Mobius-general} follows at once from Stickelberger's identity~\cite{St98} (see Lemma~\ref{lm:Stickelberger} below) linking the geometric M{\"o}bius function to quadratic characters 
of polynomial discriminants and the bound in Corollary~\ref{coro:disccharactersum-monic-range}.

Similarly, we have the following result for trinomials, implied by  Corollary~\ref{coro:disccharactersum-trinomial}.

\begin{theorem}\label{thm:Mobius-trinomial}
Let $d>k\geqslant1$ be fixed and $e=\gcd(d,k)$. For all $A_0,B_0\in \Z$ and $\gamma>0$, put 
\[
 \cT =\{X\mapsto X^d+aX^k+b:~A_0<a\leqslant A_0+p^\gamma,\
B_0<b\leqslant B_0+p^\gamma\}.
\]
For any fixed $\gamma$ with
\[
\frac{1}{2}-\frac{e}{2(2d+e)}<\gamma<1,
\]
there exists some $\delta>0$,   depending only on $d$ and $\gamma$, 
such that
\[
\fM (\cT) \ll p^{2\gamma-\delta},
\]
where the implied constant depends only on  $d$ and $\gamma$. 
\end{theorem}

The values of $\delta$ in Theorems~\ref{thm:Mobius-general} and~\ref{thm:Mobius-trinomial} can be made explicit
in terms of $d$ and $\gamma$, and the sides of the boxes therein do not necessarily have the same length. Moreover, if $e$ is even in Theorem~\ref{thm:Mobius-trinomial}, one may enlarge the range of $\gamma$ to $\frac{1}{4}<\gamma<1$. This follows from the classical Burgess bound~\eqref{eq:Burgess} for incomplete character sums, as one may see from the proof of Corollary~\ref{coro:disccharactersum-trinomial}.

Other polynomial families  (with some fixed coefficients) are certainly of interest as well and can be considered
via our method. We just illustrate our ideas in two extreme cases: where, besides the leading coefficient,  
all coefficients vary (as in $\cG$ in~\eqref{eq:F1}), and all but two coefficients are set to zero 
(as in $\cT$ in~\eqref{eq:F2}).

Note that in the family $\cF_0$ given by~\eqref{eq:F0}, one orders the polynomials by their degrees, so that the tail of the central polynomial is varying with all coefficients running over the complete $\F_p$. This situation enjoys a lot of algebraic symmetries, which are unfortunately lost in our cases  $\cG$ and $\cT$. In Section~\ref{sec: shifts}, we explain the underlying ideas in our arguments, which are built on incomplete character sums over finite fields.

\subsection{Our approach via  character sums and shifts ``{\`a} la carte''}
\label{sec: shifts}
In contrast to the geometric arguments of Sawin~\cite{Sa21,Sa24}, we study the average $\fM(\cG)$ via quadratic character sums over discriminants, based on a basic fact due to Stickelberger~\cite{St98} (see Lemma~\ref{lm:Stickelberger} below).
Roughly speaking, we work with the character sum
\begin{equation}\label{eq:discsum-generalshape}
\sum_{f\in \cF}\(\frac{\Disc(f)}{p}\),
\end{equation}
where $f$ runs over certain sets $\cF$ of polynomials with all coefficients lying in intervals of length $H$, and $\Disc(f)$ denotes the discriminant polynomial of $f$ (see~\eqref{eq:discriminant-definition} below for the precise definition).
Note that $\Disc(f)$ is a homogeneous multivariate polynomial in all coefficients of $f$, but this homogeneity
breaks down when the coefficient of the leading term is fixed as $1$. On the other hand, Lemma~\ref{lm:twistedhomogeneity} below shows a {\it twisted homogeneity} holds which turns out to be very essential in our arguments.

We would like to recall that Pierce and Xu~\cite{PX20} developed the Burgess method, and succeeded in estimating character sums over homogeneous polynomials (say, forms) below the P\'olya--Vinogradov range in a very general setting. Unfortunately in our situation, the multivariate polynomial $\Disc(f)$ is no longer homogeneous if $f$ is restricted to monic polynomials, with all other coefficients varying.
To overcome this obstacle we introduce 
a new type of amplification via non-linear shifts ``{\it{{\`a} la carte}}'',  which allows us to bound the character sum~\eqref{eq:discsum-generalshape}, and hence the sums $\fM (\cG)$ and $\fM (\cT)$ with power-savings, also
beyond the P\'olya--Vinogradov range. Indeed, the shifts ``{\it{{\`a} la carte}}'' are built on the 
{\it twisted homogeneity} of discriminants of monic polynomials as explained in Lemma~\ref{lm:twistedhomogeneity}.

\subsection{Background on character sums}

Given a polynomial $F$ with integral coefficients in $n\geqslant1$ variables, analytic number theorists quite often need to estimate  the character sum
\[
S_\chi(\cB,F) = \sum_{\bx\in\cB}\chi(F(\bx)),
\]
where $\cB$ is a box in $\F_p^n$ and $\chi$ is a non-trivial Dirichlet  
character modulo $p$.  The heart of this problem is to beat the trivial upper bound 
$|S_\chi(\cB,F)| \le \sharp\cB$ with $\cB$ as sparse as possible. Note that, in the above-mentioned applications to sums of the geometric M\"obius function, $\chi$ is quadratic, and $F$ is the discriminant of a monic polynomial.

The special case $n=\deg F=1$ has received quite a lot of attention due to many deep applications in number theory. 
For an interval $\cI \subseteq\Z$, for the sum 
\[
S_\chi(\cI)=\sum_{a\in \cI}\chi(a)
\]
a classical work by P\'olya~\cite{Pol18} and Vinogradov~\cite{Vi18} yields
\begin{equation}\label{eq:P-V}
|S_\chi(\cI)|\leqslant \sqrt{p}\log p.
\end{equation}
This is a direct consequence of Fourier analysis, and it is non-trivial as long as $\sharp\cI >p^{\frac{1}{2}+\varepsilon}$ for any $\varepsilon>0$, known as the {\it P\'olya--Vinogradov barrier}. 
Burgess~\cite{Bu57,Bu62a,Bu62b,Bu63} has introduced an ingenious method to deal with very short character sums, and proved that for each $r\in\N$, 
\begin{equation}\label{eq:Burgess}
|S_\chi(\cI)|\leqslant  \sharp\cI^{1-\frac{1}{r}}p^{\frac{r+1}{4r^2}+o(1)},
\end{equation}
which is non-trivial as long as the length   of $\cI$ satisfies $\sharp\cI>p^{\frac{1}{4}+\varepsilon}$ for any $\varepsilon>0$.

We now turn to the case $n\geqslant2$.
It should be pointed out that Burgess's method (see~\cite[Theorem~12.6]{IK04} for a modern version) relies heavily on the multiplicativity and periodicity of $\chi$, and it does not directly apply to a general inhomogeneous
polynomial with $n=1$ and $\deg F\geqslant2$, in which case one may employ the arguments of P\'olya--Vinogradov instead. Indeed, the bound~\eqref{eq:P-V}  can  easily be extended to arbitrary $n$-variate polynomials $F$ and thus, in non-degenerate cases, leading to the bound 
\begin{equation}\label{eq:PVrange}
S_\chi(\cB,F) = O(H^{n-1} \sqrt{p}\log p),   
\end{equation}
where $H$ is the largest side length of the box $\cB$. 

However, for $n\geqslant2$ and $\deg F\geqslant2$, there exist many works going beyond the P\'olya--Vinogradov barrier by imposing 
certain symmetries on $F$. Among these we would like to mention another work of Burgess~\cite{Bu68}, 
Bourgain and Chang~\cite{BC10}, Chang~\cite{Cha09}, Pierce and Xu~\cite{PX20}  and recent works by
Chu~\cite{Chu25,Chu26}. One common feature in the above works before~\cite{Chu26}
is that the polynomial is homogeneous 
in $n\geqslant2$ variables. 
It seems that~\cite{Chu26} is the first one that can go beyond the P\'olya--Vinogradov barrier in a very large framework of inhomogeneous polynomials. We also note that Chu~\cite{Chu26} achieves the power-saving bounds when the number $n$ of variables is much larger than the degree $d$ of the polynomial, say $n\geqslant 2^{d-2}(d-1)^2(d+1)$ under some geometric conditions, however, the individual  admissible ranges
of variables in~\cite{Chu26}  are strikingly short. The argument of~\cite{Chu26} is quite different and is based on the recent progress towards Vinogradov's main conjecture.

\subsection{New results about character sums}
Keeping in mind applications to sums of the geometric M\"obius function as given in Theorems~\ref{thm:Mobius-general} and~\ref{thm:Mobius-trinomial}, we study short character sums over special inhomogeneous polynomials, that is, discriminant polynomials of general monic polynomials and trinomials. In both cases, the number of variables can be considerably close to the degree of the polynomial.

Recall that for a polynomial $f(X)\in\C[X]$ of degree $d$ given by
\[
f(X)=a_dX^d+\cdots +a_1X+a_0,\quad  (a_i\in\C,~a_d\neq0),
\]
we have the formula for the {\it discriminant}
\begin{equation}\label{eq:discriminant-definition}
\Disc(f)
=a_d^{2d-2} \prod_{i<j}(r_i-r_j)^2
=(-1)^{d(d-1)/2}a_d^{2d-2}\prod_{i\neq j}(r_i-r_j),
\end{equation}
where $r_1,r_2,\cdots,r_d$ are all the roots of $f$. 
The discriminant is very important in the sense that
it reflects some properties of roots of the polynomial without computing them. 

Let $p$ be a large prime, and $\chi$ a Dirichlet character modulo $p$. Given a chosen finite family $\cF$ of polynomials, we form the character sum
\[
\fS_{\chi} (\cF) =\sum_{f\in\cF}\chi(\Disc(f)).
\]
In particular, we specialise $\cF\in\{\cG,\cT\}$ as given by~\eqref{eq:F1} and~\eqref{eq:F2}.

We now state our bound for incomplete character sums associated to the family~\eqref{eq:F1}.

\begin{theorem}\label{thm:disccharactersum-monic}
Let $h,d\geqslant2$ be fixed integers.
Let $\chi$ be a Dirichlet character modulo $p$ with order $h$. 
For all $N_1, \ldots, N_d\in \Z$ and positive integers $H_1, \ldots, H_d\leqslant p$,  put 
\[
\cB=\Z^d\cap\prod_{1\leqslant j\leqslant d}]N_j,N_j+H_j],
\]
and
\[
\cG=\{X\mapsto X^d+a_1X^{d-1}+\cdots+a_{d-1}X+a_d:~(a_1,\cdots,a_{d-1},a_d)\in\cB\}.
\]
For each integer $r\geqslant2$, put 
\[
\varTheta=\fl{(r-1)/(d-1)} \mand R=r-1-(d-2)\varTheta.
\]
Define $\varpi$ to be the smallest index $j\in\{1,\cdots,d-1,d\}$ which minimises the powers 
$\varDelta_j:=(H_jp^{-\frac{1}{2R}})^{1/j}:$
\[
\varpi
= \min\{1\leqslant j\leqslant d:~\varDelta_j\leqslant \varDelta_k\text{ for all }1\leqslant k\leqslant d\}.
\]

Then for all sufficiently large prime $p$, we have
\begin{align*}
|\fS_{\chi} (\cG)|
&\le (H_1\cdots H_d)^{1-\frac{1}{2r}}H_\varpi^{-\frac{1}{2\varpi r}}p^{\frac{d}{2r}-\frac{1}{4r}-\frac{\varpi(r-1)-1}{4\varpi rR}+o(1)}\\
&\qquad \qquad +(H_1\cdots H_d)^{\frac{1}{2}-\frac{1}{2r}}H_\varpi^{\frac{d(d+1)}{4\varpi}-\frac{1}{2\varpi r}}p^{\frac{d}{2r}-\frac{rd(d+1)-2}{8\varpi rR}+o(1)},
\end{align*}
provided that
\begin{equation}\label{eq:HjHvarpi}
H_\varpi>p^{\frac{1}{2R}}, \ \ \ \min_{1\leqslant j\leqslant d}(H_j)^\varpi H_\varpi <p^{\varpi+\frac{1}{2R}}.\end{equation}
\end{theorem}

Since it is not immediately clear where the bound of Theorem~\ref{thm:disccharactersum-monic} is non-trivial, we now present its simplified version for hypercubes. In fact, one has $\varpi=d$ in Theorem~\ref{thm:disccharactersum-monic} when  $H_1= \cdots = H_d$ therein. This observation leads to the following consequence.

\begin{corollary}\label{coro:disccharactersum-monic}
Under the conditions of  Theorem~\ref{thm:disccharactersum-monic}, if $H_1= \cdots = H_d = H$, we have
\[
|\fS_{\chi} (\cG)|
\le H^{d-\frac{d}{2r}-\frac{1}{2dr}}p^{\frac{d}{2r}-\frac{1}{4r}-\frac{d(r-1)-1}{4d rR}+o(1)}
  +H^{\frac{3d+1}{4}-\frac{d}{2r}-\frac{1}{2dr}}p^{\frac{d}{2r}-\frac{rd(d+1)-2}{8rdR}+o(1)},
\]
provided that
\[
p^{\frac{1}{2R}}<H<p^{\frac{d}{d+1}+\frac{1}{2(d+1)R}}.
\]
\end{corollary}

Taking $r$ sufficiently large with $(d-1)\mid (r-1)$, we conclude the following corollary.

\begin{corollary}\label{coro:disccharactersum-monic-range}
Under the conditions of  Theorem~\ref{thm:disccharactersum-monic}, take $H_1= \cdots = H_d=p^\gamma$. For all $\gamma$ with
\[
\frac{1}{2}-\frac{1}{2(d^2+1)}<\gamma<1,
\]
there exists some constant $\delta>0$, depending only on $d$, $h$ and $\gamma$, such that
\[
\fS_{\chi} (\cG) \ll H^d p^{-\delta},
\]
where the implied constant  
depends on $d$, $h$ and $\gamma$.
\end{corollary}

In fact, a direct consequence of Corollary~\ref{coro:disccharactersum-monic}
requires $\gamma<d/(d+1)$ in Corollary~\ref{coro:disccharactersum-monic-range}.
The validity of $d/(d+1)\leqslant \gamma<1$ can be guaranteed by~\eqref{eq:PVrange} or by splitting the original box into smaller ones, to each of them one then applies Corollary~\ref{coro:disccharactersum-monic}.
Combined with Lemma~\ref{lm:Stickelberger}, Corollary~\ref{coro:disccharactersum-monic-range} immediately implies Theorem~\ref{thm:Mobius-general}. 

We now move to bounding $\fS_{\chi} (\cT)$ with $\cT$ given by~\eqref{eq:F2}. 
Each element $f\in\cT$ is a trinomial, with many vanishing coefficients, so that $\cT$ forms a very thin family. Fortunately, the discriminant of such $f$ can be expressed explicitly in terms of the coefficients of $f$ (see Lemma~\ref{lm:trinomial-discriminant}),  reducing the problem to estimating the double character sum
\[
S_{\fs,\chi} (\cA,\cB) = \sum_{a \in \cA}\sum_{b \in \cB} \chi(a^\kappa b^\ell(sa^m + tb^n)^j)
\] 
with 
\begin{equation}\label{eq:AB-intervals}
\cA= ]A_0,A_0+A]\cap\Z \mand \cB=]B_0,B_0+B]\cap\Z,
\end{equation}
where $A_0,B_0\in\Z$, $A,B\in\Z^+$ and $\fs$ is the septuple 
\begin{equation}
\label{eq:septuple}
\fs=(j,\kappa,\ell,m,n,s,t)\in\Z^7
\end{equation}
with $\kappa,\ell\geqslant0$, $j,m,n\geqslant1$ and $p\nmid st$.

\begin{theorem}\label{thm:doublecharactersum}
Let $h\geqslant2$ be a fixed integer.
Let $\chi$ be a Dirichlet character modulo $p$ of order $h$. 
Assume that $\fs$ is in~\eqref{eq:septuple} such that 
\[
0 \le \kappa, \ell < h, \qquad  1\leqslant j<h,\qquad m,n \ge 1,   \qquad st \not \equiv 0 \bmod p.
\] 
Let $r\geqslant3$ be any positive integer such that  $\cA$ and $\cB$  in~\eqref{eq:AB-intervals} satisfy
\[
p^{\frac{1}{2(r-1)}}\leqslant A,B<p/2.
\]
Then for all sufficiently large prime $p$, we have
\[
|S_{\fs,\chi} (\cA, \cB)| \le 
(AB)^{1-\frac{1}{2r}} p^{\frac{1}{2r}+o(1)} \(\(A^{-1}p^{\frac{1}{2(r-1)}}\)^{\frac{1}{n_*}}
+\(B^{-1}p^{\frac{1}{2(r-1)}}\)^{\frac{1}{m_*}}\)^{\frac{1}{2r}}, 
\] 
where $m_*=m/\gcd(m,n)$ and  $n_*=n/\gcd(m,n)$.
\end{theorem}

To understand the strength of Theorem~\ref{thm:doublecharactersum}, we consider the most interesting 
diagonal case $A = B$. 

\begin{corollary}\label{cor:trinom A=B}
Under the conditions of  Theorem~\ref{thm:doublecharactersum}, and $m\leqslant n$, 
let $r\geqslant3$ be any positive integer such that  $\cA$ and $\cB$  in~\eqref{eq:AB-intervals} satisfy
\[
p^{\frac{1}{2(r-1)}}\leqslant A=B<p/2.
\]
Then we have
\[
|S_{\fs,\chi} (\cA, \cB)| \le 
A^{2-\frac{1}{r}-\frac{1}{2rn_*}} p^{\frac{1}{2r}+\frac{1}{4r(r-1)n_*}+o(1)}, 
\] 
where $n_*=n/\gcd(m,n)$.
\end{corollary}

Corollary~\ref{cor:trinom A=B} implies that for  a sufficiently large $r$, we have a non-trivial bound with a power saving 
for $S_{\fs,\chi} (\cA, \cB;p)$
provided 
\[
A=B \ge p^{\frac{n_*}{2n_*+1}+\varepsilon}
\]
for some fixed $\varepsilon > 0$
assuming that  $n \ge m$.

With the aid of explicit expressions of discriminants of trinomials (see Lemma~\ref{lm:trinomial-discriminant}), we may immediately translate Theorem~\ref{thm:doublecharactersum} into the estimate for $\fS_{\chi} (\cT)$ with $\cT$ given by~\eqref{eq:F2}.

\begin{corollary}\label{coro:disccharactersum-trinomial} 
Let $h\geqslant2$ be a fixed integer.
Let $\chi$ be a Dirichlet character modulo $p$ of order $h$. 
For $d>k\geqslant1$ with $e=\gcd(d,k)$, and $\cA$ and $\cB$ as in~\eqref{eq:AB-intervals}, put
\[
\cT=\{X\mapsto X^d+aX^k+b:~(a,b)\in\cA \times \cB\}.
\]
For any $\varepsilon > 0$ there exists some constant $\delta>0$, depending only on $d$ and $\varepsilon $ such that 
\[
\fS_{\chi} (\cT) \ll AB p^{-\delta}
\]
holds under either of the following hypotheses:
\begin{itemize}
\item $h\nmid e$, $\min\{AB^{1+\frac{e}{d}},~A^{1+\frac{e}{d-k}} B\} \ge p^{1+\varepsilon}$, $\max\{A,B\}<p/2$, 
\item $h\mid e$, $B>p^{\frac{1}{4}+\varepsilon}$, $\max\{A,B\}<p/2$, 
\end{itemize}
where the implied constant depends only on $d$, $h$ and $\varepsilon$.
\end{corollary}

The proof of Corollary~\ref{coro:disccharactersum-trinomial}, derived from Theorem~\ref{thm:doublecharactersum}, is given in Section~\ref{sec:proofofCorollary-disccharactersum-trinomial}.
To infer Theorem~\ref{thm:Mobius-trinomial} from Corollary~\ref{coro:disccharactersum-trinomial}, we take $\chi$ to be a quadratic character of order $h=2$.

\subsection{Comments}

Focusing on short character sums of inhomogeneous multivariate polynomials, the shifts {\it {\`a} la carte} introduced in this paper open a window for non-trivial bounds beyond the P\'olya--Vinogradov range in very general frameworks of inhomogeneous multivariate polynomials. The difficulties would fall into checking the finite translation stabilizer condition required in the 
stratification result of Xu~\cite[Theorem~1.1]{Xu20} (see also Lemma~\ref{lm:Xu}).

We emphasise that all the estimates in Theorems~\ref{thm:disccharactersum-monic}  and~\ref{thm:doublecharactersum} (and thus in their corollaries) depend on the order $h$ of $\chi$. This stems 
from the stratification result of Xu~\cite[Theorem~1.1]{Xu20}, which is 
at the heart of the argument. The argument itself does not appeal to the order of $\chi$. 
Note also that in our motivating application to sums of the geometric M{\"o}bius function, only quadratic characters are relevant. 

It is certainly natural to talk about the distribution of irreducible polynomials in  $\cG$ and $\cT$, since one usually links, over integers, the sum of M{\"o}bius values and counting primes. But this is not always the case. 
Our opinion is that the structures (or symmetries) in the target sets are quite crucial: Write $n=a_k73^k+\cdots +a_1 73+a_0$, and the vector $(a_k,\cdots,a_1,a_0)$ is then uniquely determined under the assumption $0\leqslant a_j<73$ and $a_k\neq0$. One may consider the average of the M\"obius function over integers  with 
restricted set of digits of $n$ in base 73. While the method of Maynard~\cite{May19} is likely to 
be capable of estimating such sums, it is not directly linked to the distribution of primes among such integers.
The discrepancy between the admissible ranges in our bounds on sums of the M\"obius function
and admissible ranges of~\cite{KR21,Sh91} on counting irreducible polynomials (see below) also confirms that these 
problems may become very different in some settings.
This also illustrates the importance of the algebraic structures in $\cG$ and $\cT$  when designing the shifts {\it {\`a} la carte}. Similar shifts also appeared in our recent works~\cite{FSX25,XZ25}.

We recall that in the aforementioned problem on the distribution of irreducible polynomials, Shparlinski \cite[Theorem~5]{Sh91} counted irreducible polynomials in $\cG$ as in~\eqref{eq:F1} with boxes of side length exceeding $p^{1-1/d}$, while for trinomials (as in~\eqref{eq:F2} with $k=1$)
another result of Shparlinski~\cite[Theorem~1]{Sh11} applies only to boxes of side length exceeding $\sqrt{p}$.
Very skewed boxes have been considered by Kurlberg and   Rosenzweig~\cite[Section~1.2]{KR21}, but one side 
still has length exceeding $\sqrt{p}$.

\subsection{Notation and conventions} 
We use the Landau symbol $O$ and the Vinogradov symbol $\ll$. Recall that the
assertions $U=O(V)$ and $U \ll V$  are both equivalent to the inequality $|U|\le cV$ 
with some constant $c>0$, which throughout this work may depend on the integer parameters 
$d$, $h$, $\kappa$, $\ell$, $m$, $n$ and $r$ and on the real parameter $\gamma$. 
We also write $U \asymp V$ to indicate $U \ll V \ll U$. 
For real $u$ and  $U \ge 1$ we write $u \sim U$ to denote $U \le u \le 2U$.

The floor and ceiling functions are defined as follows: $\lfloor x\rfloor$ is the greatest integer less than or equal to $x$, and $\lceil x\rceil$ the least integer greater than or equal to $x$. For a finite set $\cS$ we use $\sharp \cS$ 
to denote its cardinality. 

For each prime power $q$, denote by $\F_q$ the finite field of $q$ elements, and by ${\overline \F_q}$ the algebraic closure of $\F_q$. Denote by $\F_q[X]$ the set of all polynomials over $\F_q$.
The letter $p$ always denotes an odd prime. Denote by $\Xi_p$ and $\Xi_p^*$
the sets of all Dirichlet characters, and non-trivial characters modulo $p$, respectively. 
For a multiplicative character $\chi$ of the multiplicative group $\F_p^\times$, we define $\chi(0)=0$, so that $\chi$ can be regarded as an element in $\Xi_p$; we also make the convention $\chi(f(X)/g(X))=\chi(f(X))\overline{\chi(g(X))}$ for $f,g\in\F_p[X]$.

\subsection{AI tools}
The ideas and arguments are due solely to the authors. ChatGPT was used on October 15-18, 2026 to polish the presentation and check for errors, and the authors are responsible for any mathematical issues in this paper.

\section{Preliminaries} 

\subsection{Discriminants}
The following lemma is fundamental in our arguments, passing from geometric M\"obius function to quadratic character over discriminants. This is due to Stickelberger~\cite{St98}; see also~\cite{Ca53,Da55}.

\begin{lemma}\label{lm:Stickelberger}
Let $p$ be an odd prime and $f$ a monic polynomial of degree $d$ over $\F_p$. Then
\[
\mu_p(f)= (-1)^d \(\frac{\Disc(f)}{p}\),
\]
where $(\frac{\cdot}{p})$ denotes the Legendre symbol modulo $p$.
\end{lemma}

We now begin with the explicit formula for the discriminants of trinomials
which is well known and,  for instance, can be found in~\cite[Example~1.3.11]{Pr04}.

\begin{lemma}\label{lm:trinomial-discriminant}
For a polynomial $f(X)=X^d+aX^k+b$ with integers $d > k \ge 1$, we have
\[
\Disc(f)
=(-1)^{\frac{d(d-1)}{2}} b^{k-1}\left\{d^{\frac{d}{e}} b^{\frac{d-k}{e}}  -  (-1)^{\frac{d}{e}} (d-k)^{\frac{d-k}{e}} k^{\frac{k}{e}} a^{\frac{d}{e}}\right\}^e,
\]
where $e=\gcd(d,k)$.
\end{lemma}

From the explicit formula~\eqref{eq:discriminant-definition} for the discriminant of a polynomial we immediately conclude the
following properties.
 
\begin{lemma}\label{lm:homogeneity}
Let $f$ be a polynomial of degree $n$ given by
\[
f(X)=a_n X^n +\cdots+a_1 X+a_0,\quad  a_i\in\F_q,\quad  a_n\neq0.
\]
For $\alpha\in\F_q$ with $\alpha\neq0$, considering $f(\alpha x)$ and $\alpha f(x)$ as polynomials in $x$, we have 
\begin{align*}
&\Disc_x(f(\alpha x))=\alpha^{n(n-1)}\Disc_x(f(x)),\\
&\Disc_x(\alpha f(x))=\alpha^{2n-2}\Disc_x(f(x)). 
\end{align*}
\end{lemma}

\begin{lemma}\label{lm:twistedhomogeneity}
Let $f$ be a polynomial of degree $d$ given by
\[
f(X)=X^d+a_1X^{d-1}+\cdots+a_{d-1}X+a_d,\quad a_i\in\F_q.
\]
Considering $\Disc(f)= \Disc(a_1, \ldots, a_d)$ as a polynomial in $a_1,\cdots,a_d$, we have 
\[
\Disc(a_1+u_1 v, \ldots, a_d+u_d v^d) = v^{d(d-1)}
\Disc(a_1/v+u_1,\ldots,a_d/v^d+u_d) 
\]
for any $u_1,\cdots,u_d,v\in\F_q$ and $v\neq0$.
\end{lemma}

\begin{proof}
Write
\[
X^d+(a_1+u_1 v)X^{d-1}+\cdots+(a_d+u_d v^d)
=v^d\cdot g(X/v)
\]
with
\[
g(x)=x^d+(a_1/v+u_1)x^{d-1}+\cdots+(a_d/v^d+u_d).
\]
Then the result follows from Lemma~\ref{lm:homogeneity}.
\end{proof}

Our next  result has been established in~\cite[Lemma 3.1]{OS26} in a broader generality for polynomial discriminants 
over finite fields of odd characteristic. It immediately implies the following result for the discriminants of generic monic polynomials over $\F_q$. 

\begin{lemma}\label{lm:absolutelyirreducible}
Let $d \ge 2$ and $q$ an odd prime power.
Let $f(X)\in\F_q[X]$ be a monic polynomial of degree $d$. Then $\Disc(f)$ is an absolutely irreducible polynomial of degree $2(d-1)$ when considered as
a polynomial with the coefficients of $f$ as variables.
\end{lemma}

The following result presents a shift non-invariance of the discriminant. 
 
\begin{lemma}\label{lm:shiftinvariant}
Let $d \ge 2$ and $q$ a large odd prime power. 
For any non-zero polynomial $g(X)\in  \F_q[X]$ of degree at most $d-1$, there is a monic polynomial
$f(X)\in  \F_q[X]$ of degree $d$, such that $\Disc(f) \ne \Disc(f+g)$. 
\end{lemma}

\begin{proof}
The case $d=2$ can be verified directly, and we now assume that $d \ge 3$. 

Let $g$ be a non-zero polynomial in $\F_q[X]$ of degree at most $d-1$, and put 
\[
\sA= \{a\in\F_q:~g(a)\neq0\}. 
\] 
From Lagrange's theorem it follows that $\sharp \sA\geqslant q-d+1$. 
We also introduce the following several families of polynomials in $\F_q[X]$:
\begin{itemize}
\item $\sM$ consists of all monic polynomials in $\F_q[X]$ of degree $d$;

\item $\sF_a$ consists of all $f\in \sM$  with 
$f(X) \equiv -g(X) \bmod {(X-a)^2}$, and put
\[
\sF=\bigcup_{a\in\sA}\sF_a;
\]

\item $\sH_j$ consists of all polynomials in $\sF$ which are multiples of $h(X)^2$ for some non-constant monic polynomial $h(X)\in\F_q[X]$ with $\deg h=j$, and put
\[
\sH=\bigcup_{1\leqslant j\leqslant d/2}\sH_j.
\]
\end{itemize}
Hence
\[
\sharp \sF_a= q^{d-2},
\]
and
the inclusion-exclusion principle gives
\[
\sharp \sF=q^{d-2}\sharp \sA+O(q^{d-2})\gg q^{d-1}.
\]
We would like to show that 
\begin{equation}\label{eq:cH-upperbound}
\sharp \sH=o(q^{d-1})
\end{equation}
as soon as $d\geqslant3$.
If this is the case, we then conclude, for a sufficiently large $q$, that there exists $f_0\in\sF$ outside $\sH$. 
Clearly $f_0$ is square-free, while $f_0 + g$ is not. Hence 
\[
\Disc(f_0+g) = 0 \ne \Disc(f_0).
\]

It suffices to prove~\eqref{eq:cH-upperbound} for all $d\geqslant3$.
We first observe that $\sharp \sH_j$ is bounded by the number of triples $(k,h,a)$ of monic polynomials $k(X),h(X)\in \F_q[X]$ and $a\in\sA$ such that
\begin{equation}\label{eq:polynomialcongruence}
k(X)h(X)^2\equiv -g(X)\bmod {(X-a)^2}
\end{equation}
with $\deg k=d-2j$ and $\deg h=j$. Note that, for each $a\in\sA$, the polynomial $g(X)$ is coprime to $X-a$, which 
forces both $k(X)$ and $h(X)$ to be coprime to $X-a$.
For each given $h$ of degree $j$ and $a\in\sA$, the number of $k(X)\in \F_q[X]$ satisfying~\eqref{eq:polynomialcongruence} and of degree $d-2j$ is at most $O(1+q^{d-2j-2})$. This gives
\[
\sharp \sH\leqslant \sum_{1\leqslant j\leqslant d/2}\sharp \sH_j\ll \sum_{1\leqslant j\leqslant d/2}(1+q^{d-2j-2})\cdot q^j\cdot \sharp \sA\ll q^{\lfloor d/2\rfloor+1}+q^{d-2},
\]
from which we arrive at~\eqref{eq:cH-upperbound}
if $d\geqslant5$. 

We now turn to the case $d=3,4$. Following the above arguments for $d-2j\geqslant2$,
we then have
\begin{equation}\label{eq:cH-partialupperbound}
\sum_{1\leqslant j\leqslant d/2-1}\sharp \sH_j\ll \sum_{1\leqslant j\leqslant d/2-1}q^{d-2j-2}\cdot q^j\cdot \sharp \sA\ll q^{d-2}.
\end{equation}
It remains to estimate $\sharp \sH_j$ for $d-2j\leqslant1$, that is, $(d-1)/2\leqslant j\leqslant d/2$.

Firstly, we specialise $d=3$, and the condition $(d-1)/2\leqslant j\leqslant d/2$ is equivalent to $j=1$. 
Hence~\eqref{eq:polynomialcongruence} leads us to count the number of triples $(k,h,a)\in\F_q^2\times\sA$ such that
\[
(X-a)^2\mid ((X+k)(X+h)^2+g(X)).
\]
This is equivalent to the conditions (in $\F_q$)
\[
(a+k)(a+h)^2 +g(a) =0,\quad  (a+h)^2+2 (a+h)(a+k)+g'(a)=0.
\]
We have $g(a)\neq0$ for $a\in\sA$, which implies $(a+k)(a+h)\neq0$, so that
\begin{equation}\label{eq:k-a,h}
a+k=-\frac{(a+h)^2+g'(a)}{2(a+h)},
\end{equation}
and
\begin{equation}\label{eq:a,h}
((a+h)^2+g'(a))(a+h)-2g(a) =0.
\end{equation}
For each $a\in\sA$, the equation~\eqref{eq:a,h} is a cubic equation in $h$, which allows at most 3 solutions in $h\in\F_q$. Moreover, $a,h$ can determine a unique $k\in\F_q$ in view of~\eqref{eq:k-a,h}.
Therefore,
\[
\sharp \sH=\sharp \sH_1\leqslant 3q=o(q^{d-1})
\]
for $d=3$.

Secondly, we turn to $d=4$, and the condition $(d-1)/2\leqslant j\leqslant d/2$ is equivalent to $j=2$. 
Hence~\eqref{eq:polynomialcongruence} leads us to study the equation
\[
h(X)^2\equiv -g(X)\bmod {(X-a)^2}
\]
with $\deg h=2$, as $\deg k=0$ appearing in~\eqref{eq:polynomialcongruence}. For each given $h$, there are $O(1)$ choices for $a\in\sA$ such that
\[
(X-a)^2\mid (h(X)^2+g(X)).
\]
Hence
\[
\sharp \sH_2\ll q^2.
\]
Together with the previous bound $\sharp \sH_1\ll q^2$ in~\eqref{eq:cH-partialupperbound},
we obtain that
\[
\sharp \sH\leqslant \sharp \sH_1+\sharp \sH_2\ll q^2=o(q^{d-1})
\]
for $d=4$, completing the proof.
\end{proof}

The argument of Lemma~\ref{lm:shiftinvariant} applies over every sufficiently large finite extension of $\F_p$. Therefore, for any fixed non-zero polynomial 
$g\in\overline \F_p[X]$ with $\deg g\leqslant \deg f-1$, we have $\Disc(f) \neq \Disc(f+g)$ when both are considered 
as polynomials in the coefficients of $f\in \F_p[X]$. 
Thus the geometric translation stabilizer of $\Disc(f)$ is trivial.

\subsection{Stratification for complete character sums}

We need to bound high-dimen\-sional complete multiplicative character sums over finite fields. Roughly speaking, this should be very difficult in general, although the case of dimension one has been settled completely by Weil in 1940's; see~\cite[Chapter~11]{IK04}. 
We recall a very deep stratification theorem due to Xu~\cite{Xu20}.
To this end, define the constants $\vartheta_j$ for $0 \leq j \leq n$ as
\begin{equation}\label{eq:theta_j}
\vartheta_j= 
\begin{cases}
j \bigl\lfloor\frac{r-1}{2(n-1)}\bigr\rfloor, & \text { if } 0 \leq j \leq n-2; \\ 
\left\lfloor\frac{r-1}{2}\right \rfloor, & \text { if } j= n-1; \\ 
\lceil n r / 2\rceil, & \text { if } j=n.
\end{cases}
\end{equation}
Note that for $1\leqslant j\leqslant n-2$ the above values of $\vartheta_j$ are lower bounds on the more precise values  
in the original version of Xu~\cite[Section~1]{Xu20}.

We say that a rational function $H\in \F_p(X_1,\cdots,X_k)$ is {\it $h$-th power-free}
if the multiplicities of irreducible factors in the numerator and denominator of $H$ are at most $h-1$.

We are now ready to recall the following stratification result of Xu \cite[Theorem~1.1]{Xu20}.

\begin{lemma}\label{lm:Xu}
Let $n\geqslant2,r\geqslant1$ and $d_1,\cdots,d_r\geqslant2$ be fixed integers. Let $p$ be a prime, and $\chi_1,\cdots,\chi_r$ multiplicative characters of $\F_p^\times$. For each $1 \leq i \leq r$, we assume the data:
\begin{itemize}
\item[(i)] $d_i=\operatorname{ord} \chi_i$,
\item[(ii)] $F_i \in \F_p(\bx)=\F_p(x_1, \ldots, x_n)$ is a $d_i$-th power-free rational function of degree at most $D$ such that $\{\bm\in \overline{\F}_p^n: F_i(\bx) \equiv F_i(\bx+\bm)\}$ is finite.
\end{itemize}
Then there exist $C_1,C_2,C_3\geqslant1$, depending only on $n, r, d_1,\cdots,d_r, D$, such that for all $p>C_1$, there exist subschemes $\A_{\F_p}^{n r}=\cX_0 \supset \cX_1 \supset \cX_2 \supset \cdots \supset \cX_n$, such that the sum of degrees of irreducible components of each $X_j$ is at most $C_2$, and of co-dimension $\operatorname{codim} \cX_j \geq \vartheta_j$ and
\[
\left|\sum_{\bm \in \F_p^n} \prod_{i=1}^r \chi_i(F_i(\bm+\bx^{(i)}))\right| \leq C_3p^{(n+j-1) / 2}
\]
for each $1 \leq j \leq n$ and $(\bx^{(1)}, \ldots, \bx^{(r)}) \in \mathbb{A}_{\F_p}^{n r}\backslash \cX_j$.
\end{lemma}

Together with applications of the above stratification theorem, we usually require to bound the number of (exceptional) integral points in a box that fall into varieties over finite fields; see~\cite[Lemma~1.7]{Xu20} for the details.

\begin{lemma}\label{lm:Xu-counting}
Let $\kappa$ be a finite field, $\cX \subset \mathbb{A}_\kappa^n$  a subscheme of codimension $\vartheta$, and  $s$ the sum of the degrees of its irreducible components. If $\left\{B_i\right\}_{i=1}^n$ are subsets of $\kappa$, the ``box" $\cB=\prod_{i=1}^n \cB_i$ is naturally a subset of $\mathbb{A}^n(\kappa)$. If $1 \leq \sharp \cB_1 \leq \sharp \cB_2 \leq \cdots \leq \sharp \cB_n<\infty$, we have
\[
\sharp \(\cX(\kappa) \cap \prod_{1\leqslant i\leqslant n} \cB_i\) \leq s \prod_{\vartheta+1\leqslant i\leqslant n} \sharp \cB_i=s\cdot \sharp\cB \cdot \prod_{1\leqslant i\leqslant \vartheta} \sharp (\cB_i)^{-1}.
\]
\end{lemma}

\subsection{Counting solutions to multiplicative congruences}
We first recall the bound  of Ayyad, Cochrane and Zheng~\cite[Theorem~2]{ACZ96}, on the fourth moments
of character sums, see also~\cite{Ke17} for 
a  more precise version of~\cite{ACZ96}.

\begin{lemma}
\label{lm:ACZ}   
For any integers $Z_0$ and $ Z\ge 1$, we have  
\[
\frac{1}{p-1}  \sum_{\chi \in \Xi_p^*}  
   \left|\sum_{z= Z_0+1}^{Z_0+Z} \chi(z)\right|^4 \le 
Z^2p^{o(1)}.
  \]
\end{lemma}
Next, let 
\[
\cX=  \{X_0+1, \ldots, X_0 + X\} \mand \cY =  \{Y_0+1, \ldots, Y_0 + Y\}
\]
for some integers $X_0, X, Y_0, Y$ with $X,Y\geqslant1$. 
Given a positive integer $h$ we  bound the number $T_{h}(\cX, \cY;p)$
of solutions to the congruence
\[
x_1 y_1^h \equiv x_2 y_2^h \bmod p, \qquad x_1,x_2 \in \cX, \ y_1,y_2 \in \cY,\ p\nmid x_1x_2y_1y_2.
\]

The following bound is a modification of~\cite[Theorem~4.1]{CSZ12}, which in 
turn is based on a result of
Ayyad, Cochrane and Zheng~\cite[Theorem~2]{ACZ96}, which we have given in 
Lemma~\ref{lm:ACZ} above. 

\begin{lemma}\label{lm:ThXY}
For integers $X_0, X, Y_0, Y$ with $1\leqslant  X,Y<p$,
we have
\[
T_{h}(\cX, \cY;p)= \frac{(XY)^2}{p-1} + O(X Yp^{o(1)}+ XY^2 p^{-\frac{5}{8}+ o(1)}+X^2Yp^{-1}).
\]
\end{lemma}

\begin{proof}  We essentially repeat the proof of~\cite[Theorem~4.1]{CSZ12}. 
Using the orthogonality of Dirichlet characters, we write 
\begin{align*}
T_{h}(\cX, \cY;p)
&=\sum_{x_1,x_2 \in \cX} \sum_{y_1,y_2 \in \cY}
\frac{1}{p-1}\sum_{\chi \in\Xi_p}\chi(x_1y_1^h)\overline{\chi(x_2y_2^h)}\\
&= \frac{1}{p-1}\sum_{\chi \in \Xi_p} \left|\sum_{x \in \cX}
\chi(x)\right|^2 \left|\sum_{y \in \cY} \chi^h(y)\right|^2.
\end{align*}
The contribution to the above sum from the trivial character 
$\chi=\chi_0$ is 
\[
\frac{(X+O(1))^2(Y+O(1))^2}{p-1}=\frac{(XY)^2}{p-1}+O(XY^2p^{-1}+X^2Yp^{-1}).
\]
Taking $r=2$ in the Burgess bound~\eqref{eq:Burgess}, the contribution from the characters $\chi  \in \Xi_p^*$ with $\chi^h =  \chi_0$  can be estimated  as 
\[
\frac{Xp^{\frac{3}{8}+ o(1)}}{p-1}\sum_{\substack{\chi \in \Xi_p^*\\
\chi^h =  \chi_0}}  \left|\sum_{y \in \cY} \chi^h(y)\right|^2
\ll XY^2 p^{-\frac{5}{8}+ o(1)}. 
\]
Therefore,
\begin{equation}\label{eq:TW}
T_{h}(\cX, \cY;p)
=\frac{X^2Y^2}{p-1}+ O(XY^2 p^{-\frac{5}{8}+ o(1)}+X^2Yp^{-1} + W),
\end{equation}
where
\[
W = \frac{1}{p-1}\sum_{\substack{\chi \in\Xi_p^*\\ \chi^h  \ne \chi_0}}  \left|\sum_{x \in \cX}
\chi(x)\right|^2 \left|\sum_{y \in \cY} \chi^h(y)\right|^2.
\]

Using the Cauchy inequality, we derive
\begin{equation}\label{eq:W-Cauchy}
W^2 \le \frac{1}{(p-1)^2}\sum_{\substack{\chi \in\Xi_p^*\\ \chi^h  \ne \chi_0}}  \left|\sum_{x \in \cX}
\chi(x)\right|^4 \times
\sum_{\substack{\chi \in\Xi_p^*\\ \chi^h  \ne \chi_0}}  \left|\sum_{y \in \cY} \chi^h(y)\right|^4.
\end{equation}
When $\chi$ runs through $\Xi_p$ the power $\chi^h$ represents any other
character in $\Xi_p$ no more than $h$ times.
Thus
\[
\sum_{\substack{\chi \in\Xi_p^*\\ \chi^h  \ne \chi_0}}  \left|\sum_{y \in \cY} \chi^h(y)\right|^4\le h \sum_{\chi \in\Xi_p^*} \left|\sum_{y \in \cY} \chi(y)\right|^4.
\]
Using Lemma~\ref{lm:ACZ}, and then 
recalling~\eqref{eq:W-Cauchy} we infer
\[
W \le XY p^{o(1)}.
\]
After substituting this in~\eqref{eq:TW}, we complete the proof of the lemma.
\end{proof}

\section{Proof of Theorem~\ref{thm:disccharactersum-monic}}

\subsection{Amplification via multivariate bilinear shifts}
\label{sec:Ampl} 
For the monic polynomial
\[
f(X)=X^d+a_1X^{d-1}+\cdots+a_{d-1}X+a_d\in\Z[X],
\]
we write
\[
\Disc(f)=D(a_1, \ldots, a_d)
\]
as a polynomial in the variables $a_1, \ldots, a_d$. We are going to estimate the character sum
\[
S = \sum_{(a_1, \ldots, a_d)\in\cB}\chi(D(a_1, \ldots, a_d)),
\]
where
\[
\cB=\Z^d\cap\prod_{1\leqslant j\leqslant d}]N_j,N_j+H_j]
\]
is as given in Theorem~\ref{thm:disccharactersum-monic}.
To mimic  the indicator function of $\cB$, introduce the function $w_j$, for $1\leqslant j\leqslant d$, as
\[
w_j(x)=
\begin{cases}
\min\{x-N_j,~1,~N_j+H_j+1-x\},& \text{for~}x\in[N_j,N_j+H_j+1],\\
0,&\text{otherwise}.
\end{cases}
\]
Integration by parts gives that the Fourier transform of $w_j(x)$ satisfies the bound 
\begin{equation}\label{eq:w-fourier}
\widehat{w_j}(\xi)\ll \min\{H_j,|\xi|^{-1},|\xi|^{-2}\}.
\end{equation}
This allows us to write
\[
S=\sum_{(a_1, \ldots, a_d)\in\Z^d}w_1(a_1)\cdots w_d(a_d)\chi(D(a_1, \ldots, a_d)).
\]
For all integers $u_1,\cdots,u_d,v$, we have
\[
S=\sum_{(a_1, \ldots, a_d)\in\Z^d}w_1(a_1+u_1v)\cdots w_d(a_d+u_dv^d)\chi(D(a_1+u_1v, \ldots, a_d+u_dv^d)).
\]
Summing over $u_j\sim U_j$, $1 \le j \le d$,  and $v\sim V$,  with
\begin{equation}\label{eq:UjV1}
U_j V^j = H_j, \quad  U_j\geqslant1~(1\leqslant j\leqslant d),
\end{equation}
and
\begin{equation}\label{eq:UjV2}
\min_{1\leqslant j\leqslant d}H_jV<p,\quad  V\geqslant1,
\end{equation}
we infer
\begin{align*}
S\ll \frac{1}{U_1\ldots U_d V}\sum_{(a_1, \ldots, a_d)\in\Z^d}&\sum_{v\sim V}\Bigl|\sum_{\bu \in \cU}w_1(a_1+u_1v)\cdots w_d(a_d+u_dv^d)\\
&\qquad \qquad \qquad \times\chi(D(a_1+u_1v, \ldots, a_d+u_dv^d))\Bigr|,
\end{align*}
where the innermost summation is over all integral vectors $\bu = (u_1, \ldots, u_d)$ in the box
\[
\cU = (]U_1,2U_1] \times \ldots \times  ]U_d,2U_d])\cap \Z^d.
\]
We can utilise the Fourier inversion to separate the variables (following the technique 
of~\cite{FM98} and using~\eqref{eq:w-fourier}), getting
\[
S\ll \frac{(\log p)^d}{U_1\ldots U_d V}\sum_{(a_1, \ldots, a_d,v)\in\cD}\left|\sum_{\bu \in \cU}\varepsilon(\bu)\chi(D(a_1+u_1v, \ldots, a_d+u_dv^d))\right|
\]
with
\[
\cD=[N_1-4H_1,N_1+H_1]\times\cdots\times [N_d-2^{d+1}H_d,N_d+H_d]\times ]V,2V],
\]
where $\varepsilon(\bu) \in \C$ is of absolute value $|\varepsilon(\bu)| = 1$ and is independent from the vector 
$(a_1, \ldots, a_d,v)$.
In view of Lemma~\ref{lm:twistedhomogeneity}, we further deduce that
\[
S\ll \frac{(\log p)^d}{U_1\ldots U_d V}\sum_{(a_1, \ldots, a_d,v)\in\cD}\left|\sum_{\bu \in \cU}\varepsilon(\bu)\chi(D(a_1\overline{v}+u_1, \ldots, a_d\overline{v}^d+u_d))\right|,
\]
where $\overline{v} v\equiv1\bmod p$.

Given a vector $\bx = (x_1, \ldots, x_d)\in \Z^d$ we denote by $\rho(\bx)$ the number of 
solutions to the system of congruence equations in $(a_1, \ldots, a_d,v)\in\cD\cap\Z^{d+1}:$
\[
x_jv^j\equiv a_j\bmod p,\quad 1\leqslant j\leqslant d.
\]
Therefore,
\begin{equation}\label{eq:S-T}
S\ll \frac{(\log p)^d}{U_1\ldots U_d V}\cdot T,
\end{equation}
where
\[
T=\sum_{\bx\in \F_p^d} \rho(\bx)\left|\sum_{\bu \in \cU}\varepsilon(\bu)\chi(D(x_1+u_1, \ldots, x_d+u_d))\right|.
\]

From the H\"older inequality, we obtain for any positive integer $r$ that
\begin{equation}
\label{eq:T-Holder}
T^{2r} \le \(\sum_{\bx\in \F_p^d} \rho(\bx) \)^{2r-2}
\(\sum_{\bx\in \F_p^d} \rho(\bx)^2\) \Sigma_r, 
\end{equation}
where
\[
\Sigma_r=\sum_{\bx\in \F_p^d}\left|\sum_{\bu \in \cU}\varepsilon(\bu)\chi(D(x_1+u_1, \ldots, x_d+u_d))\right|^{2r}.
\]

Clearly, 
\begin{equation}\label{eq:rho 1st mom}
\sum_{\bx\in \F_p^d} \rho(\bx)  \ll  H_1\cdots H_dV.
\end{equation}
We choose $t$ with $1\leqslant t\leqslant d$ such that $H_t=\min_j H_j$. 
As assumed in \eqref{eq:UjV2}, we have $H_tV< p$, so that
\begin{equation}\label{eq:rho 2nd mom}
\sum_{\bx\in \F_p^d} \rho(\bx)^2
\ll \prod_{\substack{1\leqslant j\leqslant d\\ j\neq t}}H_j\cdots
 \sum_{\substack{a_t,b_t\in[N_t-2^{t+1}H_t,N_t+H_t]\\ v,w\sim V\\ a_tw^t\equiv b_tv^t\bmod p}}1 \ll  H_1\cdots H_dV p^{o(1)}, 
\end{equation}
where the last inequality follows from Lemma~\ref{lm:ThXY} and Cauchy inequality.

It remains to estimate $\Sigma_r$.
Opening the power and switching summations, we get
\[
\Sigma_r
\leqslant \sum_{\bu^{(1)}, \ldots,\bu^{(2r)}\in \cU}|\Sigma_r(\bu^{(1)}, \ldots,\bu^{(2r)})|,
\]
where for $\bu^{(i)}=(u_1^{(i)},\ldots, u_d^{(i)})$ ($1\leqslant i\leqslant 2r$),
\begin{align*}
\Sigma_r(\bu^{(1)}, \ldots,\bu^{(2r)})
&=
\sum_{\bx\in \F_p^d}  \prod_{1\leqslant i\leqslant r} \chi\(\frac{D(\bx+\bu^{(i)})}{D(\bx+\bu^{(i+r)})}\)\\
&=
\sum_{\bx\in \F_p^d}  \prod_{1\leqslant i\leqslant r} \chi\(\frac{D(x_1+u_1^{(i)} ,\ldots, x_d+u_d^{(i)})}{D(x_1+u_1^{(i+r)}, \ldots, x_d+u_d^{(i+r)})}\).
\end{align*}

\subsection{Stratification of the sums  $\Sigma_r(\bu^{(1)}, \ldots,\bu^{(2r)})$} 
Lemma~\ref{lm:Xu} is applicable since Lemmas~\ref{lm:absolutelyirreducible} and~\ref{lm:shiftinvariant} verify the conditions~(i) and~(ii) of Lemma~\ref{lm:Xu}.
We also use Lemma~\ref{lm:Xu-counting} to bound the above $d$-dimensional character sum. 

More precisely, for $1\leqslant j\leqslant d$ there exist the ``stratification'' subschemes $\cX_j$'s corresponding to Lemma~\ref{lm:Xu}, such that $\cE_j\subseteq\cX_j$ with
\[
\cE_j
=\left\{(\bu^{(1)}, \ldots,\bu^{(2r)})\in\cU^{2r}:~|\Sigma_r(\bu^{(1)}, \ldots,\bu^{(2r)})|>Cp^{\frac{d+j-1}{2}}\right\}
\]
for some positive constant $C$ depending on $d$, $h$ and $r$.
It is elementary to observe that
\[
\cU^{2r}=\bigcup_{0\leqslant j\leqslant d}(\cE_j\setminus\cE_{j+1}),
\]
where we define $\cX_0=\A_{\F_p}^{2rd}$, $\cE_0=\cU^{2r}$ and $\cX_{d+1}=\cE_{d+1}=\emptyset$.

Note that, over the set $\cE_j\setminus\cE_{j+1}$, one has the bound
\[
|\Sigma_r(\bu^{(1)}, \ldots,\bu^{(2r)})|\leqslant Cp^{\frac{d+j}{2}}.
\]
Therefore,
\begin{equation}\label{eq:Wr-stratification}
\Sigma_r \ll \sum_{0\leqslant j\leqslant d}  \sharp (\cE_j\setminus\cE_{j+1})\ \cdot p^{\frac{d+j}{2}}\leqslant \sum_{0\leqslant j\leqslant d}  \sharp \cE_j \cdot p^{\frac{d+j}{2}}.
\end{equation}

Applying~\eqref{eq:theta_j} with $(n,r)\leftarrow (d,2r)$, we obtain the control of the co-dimensions of $\cX_j$: $\operatorname{codim} \cX_j \geq \widetilde \vartheta_j$ with
\[
\widetilde \vartheta_j=
\begin{cases}
j \varTheta, & \text { if } 1 \leq j \leq d-2; \\ 
r-1, & \text { if } j=d-1; \\ 
dr, & \text { if } j=d,
\end{cases}
\]
in which we recall that $\varTheta=\fl{(r-1)/(d-1)}$ as introduced in Theorem~\ref{thm:disccharactersum-monic}.

We now find all indices $j$ with $1\leqslant j\leqslant d$, which minimize $U_j$. The smallest one is defined as $\varpi$. As one can see later, this choice coincides with that in Theorem~\ref{thm:disccharactersum-monic} after fixing $V$ subject to~\eqref{eq:UjV2}.
Therefore, from Lemmas~\ref{lm:Xu} and~\ref{lm:Xu-counting}, it follows that
\begin{equation}\label{eq:Xj-upperbound1}
\sharp \cE_{d-1}\leqslant \sharp (\cX_{d-1}\cap \cU^{2r})\ll (U_1\cdots U_d)^{2r} U_\varpi^{1-r},
\end{equation}
and
\begin{equation}\label{eq:Xj-upperbound2}
\sharp \cE_j \leqslant \sharp (\cX_j\cap \cU^{2r})\ll (U_1\cdots U_d)^{2r} U_\varpi^{-j\varTheta},\qquad 0\leq j \leq d-2.
\end{equation}
We also need the following bound for $\sharp \cE_d$.
\begin{lemma}
With the above notation, we have
\begin{equation}\label{eq:Xd-upperbound}
\sharp \cE_d\ll (U_1\cdots U_d)^r.
\end{equation}
\end{lemma}

\begin{proof}
This does not follow directly from Lemma~\ref{lm:Xu-counting}, which is not strong enough when the edges of the box have different lengths from each other. 
We instead use the absolute irreducibility of the discriminant polynomial $D$ and the triviality of its translation stabilizer.
In fact, according to \cite[Lemma 1.14]{Xu20}, one has the Weil bound
\[
|\Sigma_r(\bu^{(1)}, \ldots,\bu^{(2r)})|\ll p^{d-\frac{1}{2}},
\]
provided the rational function
\[
R(\bx;\bu^{(1)}, \ldots,\bu^{(2r)}):=\frac{\prod_{1\leqslant i\leqslant r}D(\bx+\bu^{(i)})}{\prod_{1\leqslant i\leqslant r}D(\bx+\bu^{(i+r)})}
\]
in $\bx$ is not an $h$-th power of any other rational functions over $\F_p$, up to a constant scalar.
To bound $\sharp \cE_d$, it suffices to count the number of $\bu^{(i)}$'s such that $R(\bx;\bu^{(1)}, \ldots,\bu^{(2r)})$ is an $h$-th power.

For distinct vectors $\bv,\bw$, the polynomials $D(\bx+\bv)$ and
$D(\bx+\bw)$ are distinct, non-associate, absolutely irreducible polynomials.
If they were associates, comparison of their leading homogeneous parts would force
the proportionality constant to be 1, contradicting the absence of a nonzero
translation fixing $D$.
We now assume that $R(\bx;\bu^{(1)}, \ldots,\bu^{(2r)})$ is an $h$-th power.
For each distinct vector value $\bv$, denote by $a_{\bv}$ and $b_{\bv}$ its
multiplicities in the numerator and denominator.
Unique factorisation gives $a_{\bv}\equiv  b_{\bv}\pmod h$.
If a vector occurred only once among all $2r$ positions, this difference would be
$1$ or $-1$, which cannot be divisible by $h\ge2$. Consequently, each distinct vector
occurs at least twice, and there are at most $r$ distinct vectors.
This completes the proof of the lemma by permuting all possible appearances of these vectors in $2r$ positions.
\end{proof}

We now conclude from~\eqref{eq:Wr-stratification}, \eqref{eq:Xj-upperbound1}, \eqref{eq:Xj-upperbound2} and~\eqref{eq:Xd-upperbound} that
\begin{align*}
\Sigma_r
& \ll \sum_{0\leqslant j\leqslant d-2}(U_1\cdots U_d)^{2r} U_\varpi^{-j\varTheta}\cdot p^{\frac{d+j}{2}}\\
& \qquad \qquad \qquad \qquad  +(U_1\cdots U_d)^{2r} U_\varpi^{1-r}\cdot p^{d-\frac{1}{2}}+(U_1\cdots U_d)^r\cdot p^{d}\\
& \ll (U_1\cdots U_d)^{2r}p^{\frac{d}{2}} \fU +(U_1\cdots U_d)^r\cdot p^d,
\end{align*}
where 
\[
\fU = 1+U_\varpi^{-(d-2)\varTheta}\cdot p^{\frac{d}{2}-1}+U_\varpi^{1-r}\cdot p^{\frac{d-1}{2}}
\]
 
\subsection{Optimisation}
Take 
\[
U_\varpi=p^{\frac{1}{2R}}\mand R:=r-1-(d-2)\varTheta
\]
to balance the second and third terms in $\fU$, which then yields 
\begin{align*}
\fU & \ll 1+ p^{\frac{-(d-2)\varTheta}{2R} + \frac{d}{2}-1}
= 1+ p^{\frac{1}{2} - \frac{r-1}{2R} + \frac{d}{2}-1} \\
& = 1+ p^{ -\frac{r-1}{2R}+\frac{d-1}{2}}
\ll p^{-\frac{r-1}{2R} + \frac{d-1}{2}} 
\end{align*}
since $\varTheta\leqslant (r-1)/(d-1)$ we have 
\[
 \frac{d-1}{2} \ge \frac{r-1}{2R}.
 \] 
 Therefore,
\[
\Sigma_r
 \ll (U_1\cdots U_d)^{2r}\cdot p^{d-\frac{1}{2}-\frac{r-1}{2R}}+(U_1\cdots U_d)^r\cdot p^d,
\]
from which and also from~\eqref{eq:T-Holder}, \eqref{eq:rho 1st mom} and~\eqref{eq:rho 2nd mom}, we derive that
\[
T\le (H_1\cdots H_dV)^{1-\frac{1}{2r}}\(U_1\cdots U_d\cdot p^{\frac{d}{2r}-\frac{1}{4r}-\frac{r-1}{4rR}} +(U_1\cdots U_d)^\frac{1}{2}\cdot p^{\frac{d}{2r}}\)p^{o(1)}.
\]
From~\eqref{eq:S-T} it follows that 
\begin{equation}\label{eq:S-intermediate} 
|S|\le  (H_1\cdots H_d)^{1-\frac{1}{2r}}\(p^{\frac{d}{2r}-\frac{1}{4r}-\frac{r-1}{4rR}}+(U_1\cdots U_d)^{-\frac{1}{2}}\cdot p^{\frac{d}{2r}}\)V^{-\frac{1}{2r}} p^{o(1)}.
\end{equation}
We now take
\begin{equation}\label{eq:V}
V=\left(H_\varpi/U_\varpi\right)^{\frac{1}{\varpi}}
=(H_\varpi p^{-\frac{1}{2R}})^{\frac{1}{\varpi}}.
\end{equation}
Upon this choice, we find $V>1$ by the hypothesis~\eqref{eq:HjHvarpi}.
Together with~\eqref{eq:UjV1} this yields
\[
U_1\cdots U_d\asymp H_1\cdots H_dV^{-\frac{d(d+1)}{2}}\asymp H_1\cdots H_d\cdot H_\varpi^{-\frac{d(d+1)}{2\varpi}}p^{\frac{d(d+1)}{4\varpi R}}.
\]
Returning to~\eqref{eq:S-intermediate} we have
\begin{align*}
|S|
&\le  (H_1\cdots H_d)^{1-\frac{1}{2r}}H_\varpi^{-\frac{1}{2\varpi r}}p^{\frac{d}{2r}-\frac{1}{4r}-\frac{\varpi(r-1)-1}{4\varpi rR}+o(1)}\\
&\qquad  \quad+(H_1\cdots H_d)^{1-\frac{1}{2r}}(H_1\cdots H_d\cdot H_\varpi^{-\frac{d(d+1)}{2\varpi}}p^{\frac{d(d+1)}{4\varpi R}})^{-\frac{1}{2}}H_\varpi^{-\frac{1}{2\varpi r}}p^{\frac{d}{2r}+\frac{1}{4\varpi rR}+o(1)}\\
&\le  (H_1\cdots H_d)^{1-\frac{1}{2r}}H_\varpi^{-\frac{1}{2\varpi r}}p^{\frac{d}{2r}-\frac{1}{4r}-\frac{\varpi(r-1)-1}{4\varpi rR}+o(1)}\\
&\qquad  \quad+(H_1\cdots H_d)^{\frac{1}{2}-\frac{1}{2r}}H_\varpi^{\frac{d(d+1)}{4\varpi}-\frac{1}{2\varpi r}}p^{\frac{d}{2r}-\frac{rd(d+1)-2}{8\varpi rR}+o(1)}.
\end{align*}

To guarantee the first restriction in~\eqref{eq:UjV2}, it suffices to add
\[
\min_{1\leqslant j\leqslant d}H_j \{H_\varpi p^{-\frac{1}{2R}}\}^{\frac{1}{\varpi}}<p
\]
in view of~\eqref{eq:V},
which is exactly~\eqref{eq:HjHvarpi}. We conclude the proof of Theorem~\ref{thm:disccharactersum-monic} by verification that the chosen index $\varpi$ minimizes $U_j$. In fact, we infer from~\eqref{eq:V} and~\eqref{eq:UjV1} that
\[
U_j=H_j\{H_\varpi p^{-\frac{1}{2R}}\}^{-\frac{j}{\varpi}},\ \ 1\leqslant j\leqslant d.
\]
In view of the above choice of $\varpi$, it suffices to verify that
\[
H_j\{H_\varpi p^{-\frac{1}{2R}}\}^{-\frac{j}{\varpi}}=U_j\geqslant U_\varpi=p^{\frac{1}{2R}}
\]
for all $1\leqslant j\leqslant d$, which is equivalent to
\[
(H_jp^{-\frac{1}{2R}})^{\frac{1}{j}}\geqslant 
(H_\varpi  p^{-\frac{1}{2R}})^{\frac{1}{\varpi}}
\]
as desired in the definition of $\varpi$ in Theorem~\ref{thm:disccharactersum-monic}.
We now complete the proof.

\section{Proof of Theorem~\ref{thm:doublecharactersum}}
\subsection{Initial reduction}
We first introduce
\[
\varphi(X,Y)=X^\kappa Y^\ell(sX^m +tY^n)^j,
\]
so that
\[
S_{\fs,\chi} (\cA,\cB) = \sum_{a \in \cA}\sum_{b \in \cB} \chi(\varphi(a,b)).
\] 
Throughout this section we fix $r\geqslant3$ to be an integer.
It is convenient to assume that 
\begin{equation}
\label{eq:smallAB}
A, B \le p^{\frac{1}{2}+\frac{r-2}{2(r-1)(2r-3)}}
\end{equation} 
as otherwise the Weil bound (see~\cite[Theorem~11.23]{IK04}) augmented with the completing technique (see~\cite[Section~12.2]{IK04}) imply the bound 
\[
S_{\fs,\chi} (\cA, \cB)
\ll \min\{A,B\} p^{\frac{1}{2}}\log p+A+B
\le AB p^{-\frac{r-2}{2(r-1)(2r-3)}}\log p+A+B,
\]
which is stronger than the desired bound in Theorem~\ref{thm:doublecharactersum}. 
Note that the term $A+B$ appears to deal with the possible case that the fibres $\varphi(X_0,Y)$ and $\varphi(X,Y_0)$ are degenerate for some finitely many $X_0,Y_0$, to which the trivial bound applies.

Note that the easy case $\kappa=\ell=0$ and $m=n=1$ gives
\[
S_{\fs,\chi} (\cA,\cB) = \sum_{a \in \cA}\sum_{b \in \cB} \chi^j(sa + tb).
\] 
This reduces the problem to bounding linear incomplete character sums. Note that $\chi^j$ is still a non-trivial character modulo $p$ since $1\leqslant j<h$. Taking $r=3$ in~\eqref{eq:Burgess}, we find
\begin{equation}\label{eq:Sbound-easycase}
|S_{\fs,\chi} (\cA,\cB)|
\leqslant \min\{A^{\frac{2}{3}}B,AB^{\frac{2}{3}}\} p^{\frac{1}{9}+o(1)}\leqslant (A+B)^{-\frac{1}{3}}ABp^{\frac{1}{9}+o(1)}.
\end{equation}
In fact, the bound in Theorem~\ref{thm:doublecharactersum} shows
\[
|S_{\fs,\chi} (\cA, \cB)| \le 
(AB)^{1-\frac{1}{r}}(A+B)^{\frac{1}{2r}}p^{\frac{1}{2r}+\frac{1}{4r(r-1)}+o(1)},
\] 
which is trivial unless
\[
AB(A+B)^{-\frac{1}{2}}>p^{\frac{1}{2}+\frac{1}{4(r-1)}+o(1)},
\] 
in which case one may check~\eqref{eq:Sbound-easycase} is stronger than the desired bound in Theorem~\ref{thm:doublecharactersum}. 
We henceforth assume that
\begin{equation}\label{eq:klmn}
(\kappa,\ell,m,n)\neq (0,0,1,1).
\end{equation}

\subsection{Amplification via trivariate bilinear shifts}
We fix three  parameters $U$, $V$ and $W$ with 
\begin{equation}\label{eq: UWnVWm}
UW^{n_*}=\frac{1}{4} A,\qquad 
VW^{m_*}=\frac{1}{4} B,\qquad 
U, V, W\ge 1.
\end{equation}

Exactly as in  Section~\ref{sec:Ampl}  we apply  shifts by three variables
\[
(a,b)\mapsto (a+uw^{n_*} ,b+vw^{m_*})
\]
with $u\sim U,v\sim V$ and $w\sim W$.
Thus, by a similar argument, we have 
\begin{align*}
S_{\fs,\chi} (\cA, \cB) 
&\ll \frac{(\log p)^2}{UVW} 
\sum_{\substack{a \in \cA_1\\ p\nmid a}} \sum_{b \in \cB_1} \sum_{w \sim W}\Bigl|\sum_{u \sim U} \sum_{v \sim V}  \varepsilon(u,v)\chi(\varphi(a+uw^{n_*},b+vw^{m_*}))\Bigr|+B,
\end{align*}
where $\cA_1,\cB_1$ are certain intervals of consecutive integers of length $O(A)$ and $O(B)$, respectively, and $\varepsilon(u,v)$ is some complex weight with $| \varepsilon(u,v)| = 1$, independent from $a,b,w$.
Since $mn_*  = m_*n$, we see that 
\begin{equation}\label{eq:SS-T}
S_{\fs,\chi} (\cA, \cB) \ll \frac{(\log p)^2}{UVW} 
\sum_{\substack{a \in \cA_1\\ p\nmid a}} \sum_{b \in \cB_1} \sum_{w \sim W}
\Bigl|\sum_{u \sim U} \sum_{v \sim V}\varepsilon(u,v)
  \chi(\varphi(a \overline w^{n_*}+u,b \overline w^{m_*}+v))\Bigr|+B.
\end{equation}

\subsection{Grouping variables}
For $x,y\in\Z$, define
\begin{align*}
\sigma(x,y) 
& =\sharp \{(a,b,w)\in\cA_1\times\cB_1\times[W,2W]:\\
& \qquad \qquad  \qquad  \qquad ~a \overline w^{n_*}\equiv x\bmod p, \ b \overline w^{m_*}\equiv y\bmod p,~p\nmid a\}. 
\end{align*}
Therefore, we infer from~\eqref{eq:SS-T} that
\begin{align*}
S_{\fs,\chi}(\cA, \cB)
&\ll \frac{(\log p)^2}{UVW}
\sum_{x,y\bmod p}\sigma(x,y)
\Bigl| \sum_{u \sim U} \sum_{v \sim V}\varepsilon(u,v)\chi(\varphi(x+u,y+v))\Bigr|+B.
 \end{align*} 

By the H\"older inequality, we get
\begin{equation}\label{eq:Holder}
S_{\fs,\chi} (\cA, \cB)
\ll \frac{(\log p)^2}{UVW} \fS_1^{1-\frac{1}{r}}(\fS_2\fS_3)^{\frac{1}{2r}}+B
\end{equation} 
with
\[
\fS_i=\sum_{x,y\bmod p} \sigma(x,y)^i,\qquad i=1,2,
\]
and
\[
\fS_3=\sum_{x,y\bmod p}\Bigl| \sum_{u \sim U} \sum_{v \sim V}\varepsilon(u,v)\chi(\varphi(x+u,y+v))\Bigr|^{2r}.
\]
 
Trivially we have
\begin{equation}\label{eq:fS1-upperbound}
\fS_1\ll ABW. 
\end{equation} 
Moreover, $\fS_2$ is bounded by the number of tuples $(a_1,a_2,b_1,b_2,w_1,w_2)\in\cA_1^2\times\cB_1^2\times[W,2W]^2$ such that
\[
a_1w_2^{n_*}\equiv a_2w_1^{n_*}\bmod p,\qquad b_1w_2^{m_*}\equiv b_2w_1^{m_*}\bmod p,\qquad p\nmid a_1a_2.
\]
For any given $(a_1,a_2,w_1,w_2)\in\cA_1^2\times[W,2W]^2$, there are at most $O(B)$ choices for the pair $(b_1,b_2)\in\cB_1^2$. With the aid of Lemma~\ref{lm:ThXY} the bound
\begin{equation}
\fS_2 \ll \(\frac{A^2W^2}{p} + AW+ AW^2 p^{-\frac{5}{8}}\) Bp^{o(1)}
\ll ABWp^{o(1)},\label{eq:fS2-upperbound}
\end{equation}
holds since $W\ll p^{\frac{5}{8}}$ in view of~\eqref{eq:smallAB} and we now additionally assume that 
\begin{equation}
\label{eq:smallAW}
AW \le p
 \end{equation} 
 
Now it remains to estimate $\fS_3$, which is the target of the next subsection.

\subsection{Stratifications}
Opening the $2r$-th power and switching  summations, we find
\[
\fS_3
\leqslant \sum_{(\vu,\vv)\in\cB} | \fS_r(\vu, \vv)|,
\]
where $\cB = [U,2U]^{2r} \times [V,2V]^{2r}$, and for $\vu = (u_1, \ldots, u_{2r})$,  $\vv = (v_1, \ldots, v_{2r})$ we have
\[
\fS_r(\vu, \vv)
=\sum_{x,y\bmod p}\prod_{i=1}^r  \chi
\(\frac{\varphi(x+u_i,y+v_i)}{\varphi(x+u_{r+i},y+v_{r+i})} \). 
\]

We now appeal to a stratification result of Xu~\cite[Theorem~1.1]{Xu20}, which 
we present in the very special case relevant to the sums $ \fS_r(\vu, \vv)$. 
Namely, the parameters $(n,r,\vm)$ of Lemma~\ref{lm:Xu} are now replaced in the following way
\[
(n,r,\vm) \leftarrow (2, 2r, (x,y)),
\]
which we apply to rational functions
\[
 F_i(\vm)\leftarrow  \varphi(x,y)=x^\kappa y^\ell(sx^m +ty^n)^j
\]
(which are now all equal),  shifted by $\vx^{(i)} = (u_\nu, v_\nu)$, $\nu=1, \ldots, 2r$.

To apply Lemma~\ref{lm:Xu}, we need to check the following two facts about the function $\varphi(X,Y)=X^\kappa Y^\ell(sX^m +tY^n)^j$ in $X$ and $Y$, under the assumed conditions $0 \le \kappa,\ell < h$, $1 \le j < h$
and 
$p\nmid st$ and also provided $p$ is large enough (so that $p> mn$):
\begin{enumerate}[(i)]
\item $\varphi$ is $h$-th power-free, 
\item If $\varphi(X,Y)=\varphi(X+a,Y+b)$ identically for $X,Y\in\overline{\F}_p$, then $a=b=0$ unless $(\kappa,\ell,m,n)=(0,0,1,1)$, in which case the above shift implies $sa+tb=0$.
\end{enumerate}

To prove (i), for $P(X,Y)=sX^m +tY^n$, we observe that $\gcd(P,X)=\gcd(P,Y)=1$ and $\gcd(P,\partial_XP)=\gcd(P,smX^{m-1})=1$.
Thus $P$ is square-free and $X,Y,P$ are pairwise coprime. The irreducible-factor multiplicities
in $X^\kappa Y^\ell P^j$ are consequently $\kappa,\ell,j$, each of which is smaller than $h$.
Hence $\varphi$ must be $h$-th power-free.

We now turn to (ii). Take $p$ to be large enough such that $p>\deg\varphi$. We examine the translation $(a,b)$ such that
\begin{equation}\label{eq:varphi(X,Y)-shift}
\varphi(X,Y)\equiv \varphi(X+a,Y+b)
\end{equation}
as polynomials in $X,Y$.

We first consider the case $\kappa>0$. We now find $\varphi(0,Y)\equiv \varphi(a,Y+b)$, as polynomials in $Y$. More precisely,
$a^\kappa(Y+b)^\ell(sa^m +t(Y+b)^n)^j=0$ for all $Y\in\overline \F_p$. If $a\neq0$, 
we have $(Y+b)^\ell(sa^m +t(Y+b)^n)^j=0$ for all $Y\in\overline \F_p$, which is obviously impossible, hence this implies $a=0$.
Cancelling the factor $X^\kappa$, we obtain $Y^{\ell}(sX^m+tY^n)^j\equiv (Y+b)^\ell(sX^m +t(Y+b)^n)^j$. Taking $X=0$ this yields $Y^{\ell+jn}\equiv (Y+b)^{\ell+jn}$ which forces $b=0$. Similarly, if $\ell>0$,~\eqref{eq:varphi(X,Y)-shift} also implies $a=b=0$.

If $\kappa=\ell=0$, then $P(X+a,Y+b)^j\equiv P(X,Y)^j$ implies
$P(X+a,Y+b)\equiv cP(X,Y)$ for some $c$ with $c^j=1$. Comparison of the leading
homogeneous parts gives $c=1$, so that $P(X+a,Y+b)\equiv P(X,Y)$. Assume that $(a,b)\neq(0,0)$ in~\eqref{eq:varphi(X,Y)-shift}. Iterating with $(X,Y)\rightarrow(X+a,Y+b)$, we find $P(X,Y)\equiv P(X+ua,Y+ub)$ for all $u\in\F_p$, that is, $sX^m +tY^n\equiv s(X+ua)^m +t(Y+ub)^n$. Taking the derivative with respect to $u$ on both sides, we get $smaX^{m-1}+tnbY^{n-1}\equiv 0$ since $p$ is large enough. This pushes $m=n=1$ and $sa+tb=0$.

Under the assumption~\eqref{eq:klmn} we can apply Lemma~\ref{lm:Xu}  with 
\[
(\vartheta_0, \vartheta_1, \vartheta_2) \leftarrow (0, r-1,  2r), 
\]
and thus  with $ \cX_0,  \cX_1,  \cX_2$ of dimensions satisfying
  \begin{align*} 
&  \dim \cX_0\leq 4r-\vartheta_0 =4r, \\   
&\dim \cX_1\leq 4r -\vartheta_1  =3r+1,\\
&\dim \cX_2 \leq 4r -\vartheta_2 = 2r,
 \end{align*}   
 and such that  one has the inequalities
\[
\fS_r (\vu, \vv)
\ll 
\begin{cases} 
p^2,\  &\text{for }(\vu, \vv) \in \cX_2,\\
p^{\frac{3}{2}},  &\text{for }(\vu, \vv) \in \cX_0 \setminus \cX_2,\\
p,  &\text{for }(\vu, \vv) \in \cX_0 \setminus \cX_1.
\end{cases}
\]
Therefore,
\begin{equation}\label{eq:Sigma Strat}
\fS_3
\ll \sharp(\cX_2 \cap  \cB) \cdot p^2 +  \sharp(\cX_1 \cap  \cB)   \cdot p^{\frac{3}{2}}  + \sharp(\cX_0 \cap  \cB)   \cdot p .
\end{equation}
Invoking Lemma~\ref{lm:Xu-counting}, we see that   the inequality~\eqref{eq:Sigma Strat}  yields  
\begin{equation} \label{eq:fS3-upperbound}
 \begin{split}
\fS_3
&\ll  p^2 \max\{U,V\}^{2r}+  p^{\frac{3}{2}} \max\{U,V\}^{2r} \min\{U,V\}^{r+1} +p (UV)^{2r}\\
&\ll  p^2(U+V)^{2r}+  p^{\frac{3}{2}} (U+V)^{r-1}(UV)^{r+1} +p (UV)^{2r} \\
&\ll  p^2(UV)^{2r}(\min\{U,V\}^{-2r}+  p^{-\frac{1}{2}}\min\{U,V\}^{1-r}+p^{-1}).
\end{split}
 \end{equation}

\subsection{Optimisation} 
Substituting the bounds~\eqref{eq:fS1-upperbound}, \eqref{eq:fS2-upperbound} and~\eqref{eq:fS3-upperbound} in~\eqref{eq:Holder}, we obtain 
\begin{equation} \label{eq:Sbound 1}
 \begin{split}
 &|S_{\fs,\chi} (\cA, \cB)|\\
 & \quad  \le  (AB)^{1- \frac{1}{2r}} W^{-\frac{1}{2r}}(\min\{U,V\}^{-2r}+  p^{-\frac{1}{2}}\min\{U,V\}^{1-r}+p^{-1})^{\frac{1}{2r}} p^{\frac{1}{r}+ o(1)}+B.
\end{split}
 \end{equation}
 
We choose $U$ and $V$ in such a way that 
\[
\min\{U,V\}=\frac{1}{4} p^{\frac{1}{2(r-1)}}
\]
to balance the terms $p^{-\frac{1}{2}}\min\{U,V\}^{1-r}$ and $p^{-1}$ in~\eqref{eq:Sbound 1}. This gives
\[
 |S_{\fs,\chi} (\cA, \cB)|
  \le  (AB)^{1- \frac{1}{2r}} W^{-\frac{1}{2r}} p^{\frac{1}{2r}+o(1)}+B.
\]
We now return to~\eqref{eq: UWnVWm} to determine the choice of $W$.

{\bf Case I:} Suppose $U\geqslant V$, then we have $V=\frac{1}{4} p^{\frac{1}{2(r-1)}}$, and choose
\begin{equation}\label{eq:U,W1}
W=(B p^{-\frac{1}{2(r-1)}})^{\frac{1}{m_*}} \mand
U=\frac{1}{4}A(B p^{-\frac{1}{2(r-1)}})^{-\frac{n}{m}},
\end{equation} 
subject to the constraints
\[
B>p^{\frac{1}{2(r-1)}} \mand (A p^{-\frac{1}{2(r-1)}})^m\geqslant (B p^{-\frac{1}{2(r-1)}})^n.
\]
These guarantee $U,V,W\geqslant1$ and $U\geqslant V$ as desired.

{\bf Case II:} Suppose $U< V$, then we have $U=\frac{1}{4} p^{\frac{1}{2(r-1)}}$, and choose
\begin{equation}\label{eq:U,W2}
W=(A p^{-\frac{1}{2(r-1)}})^{\frac{1}{n_*}} \mand 
V=\frac{1}{4}B(A p^{-\frac{1}{2(r-1)}})^{-\frac{m}{n}},
\end{equation} 
subject to the constraints
\[
A>p^{\frac{1}{2(r-1)}} \mand (A p^{-\frac{1}{2(r-1)}})^m< (B p^{-\frac{1}{2(r-1)}})^n.
\]
These guarantee $U,V,W\geqslant1$ and $U<V$ as desired.

Note that the condition~\eqref{eq:smallAW} holds in view of the above choices and~\eqref{eq:smallAB}.
Therefore, under the condition
\[
A,B>p^{\frac{1}{2(r-1)}},
\]
we find
\[
(A p^{-\frac{1}{2(r-1)}})^m \ge (B p^{-\frac{1}{2(r-1)}})^n \Longleftrightarrow
(A p^{-\frac{1}{2(r-1)}})^{\frac{1}{n_*}}\ge (B p^{-\frac{1}{2(r-1)}})^{\frac{1}{m_*}}.
\]
Combining the above two cases, we derive from~\eqref{eq:Sbound 1}, \eqref{eq:U,W1} and~\eqref{eq:U,W2} that
\begin{align*}
|S_{\fs,\chi} (\cA, \cB)|
&\le (AB)^{1-\frac{1}{2r}} p^{\frac{1}{2r}+o(1)} \min\{(Ap^{-\frac{1}{2(r-1)}})^{\frac{1}{n_*}},(Bp^{-\frac{1}{2(r-1)}})^{\frac{1}{m_*}}\}^{-\frac{1}{2r}}+B\\
&\le (AB)^{1-\frac{1}{2r}} p^{\frac{1}{2r}+o(1)} \{(A^{-1}p^{\frac{1}{2(r-1)}})^{\frac{1}{n_*}}+(B^{-1}p^{\frac{1}{2(r-1)}})^{\frac{1}{m_*}}\}^{\frac{1}{2r}}.
\end{align*} 
This completes the proof of Theorem~\ref{thm:doublecharactersum}.

\section{Deriving Corollary~\ref{coro:disccharactersum-trinomial} from Theorem~\ref{thm:doublecharactersum}}
\label{sec:proofofCorollary-disccharactersum-trinomial}

With the aid of explicit expressions of discriminants of trinomials (see Lemma~\ref{lm:trinomial-discriminant}), we may immediately translate Theorem~\ref{thm:doublecharactersum} into the estimate for $\fS_{\chi} (\cT)$ with $\cT$ given by~\eqref{eq:F2}. 
Indeed, we may write
\[
\fS_{\chi} (\cT)=\chi(-1)^{\frac{d(d-1)}{2}}\sum_{a\in\cA}\sum_{b\in\cB}\chi(b)^{k-1}\chi^e(d^{\frac{d}{e}} b^{\frac{d-k}{e}}- (-1)^{\frac{d}{e}} (d-k)^{\frac{d-k}{e}} k^{\frac{k}{e}} a^{\frac{d}{e}}),
\]
where $e=\gcd(d,k)$. If $h=\ord(\chi)\mid e$, then $\chi^e$ becomes a trivial character modulo $p$, but $h\nmid(k-1)$, which reduces the problem to bounding the incomplete character sum
\[
\sum_{b\in\cB}\chi^{k-1}(b)
\]
with $b$ restricted by $d^{\frac{d}{e}} b^{\frac{d-k}{e}}\not\equiv (-1)^{\frac{d}{e}} (d-k)^{\frac{d-k}{e}} k^{\frac{k}{e}} a^{\frac{d}{e}}\bmod p$. This congruence condition is harmless since, for each fixed $a$, it rules out finitely many possible values of $b$.
Now Burgess's bound~\eqref{eq:Burgess} applies.
If $h\nmid e$, then $\chi^e$ is still a non-trivial character modulo $p$. 
We may apply Theorem~\ref{thm:doublecharactersum} with
\[
(\kappa,\ell,m=m^*,n=n^*,s,t,j)\leftarrow \(0,k_h-1,\frac{d}{e},\frac{d-k}{e},- (-1)^{\frac{d}{e}} (d-k)^{\frac{d-k}{e}} k^{\frac{k}{e}},d^{\frac{d}{e}},e_h\).
\]
where $e_h$ and $k_h$ are the smallest positive residues of $e$ and $k$ modulo 
$h$, respectively.
The restrictions on sizes of $A$ and $B$ can be verified directly.

\section{Acknowledgements}
We would like to thank Olivier Wittenberg for a very useful discussion of the shift non-invariance of the discriminant given in  Lemma~\ref{lm:shiftinvariant}.

This work was completed during the Program ``Fouvry-73'' held in the Bernoulli Centre of EPF Lausanne during the summer of 2026. We thank the organisers of this four-week program for their hospitality and wonderful working environment. I.S. would also like to thank the Max Planck Institute for 
Mathematics, Bonn, for its support and excellent working conditions during the
preparation of this work.

I.S. is supported in part by the Australian Research Council Grant  DP230100534, and P.X. is supported in part by NSFC (No. 12625101), Shaanxi NSF (No. 2025JC-QYCX-002) and Shaanxi Fundamental Science Research Project for Mathematics and Physics (No.25JSZ007).

\end{document}